\documentclass[12pt,a4paper,reqno]{amsart}
\usepackage{amsmath}
\usepackage{amsfonts}
\usepackage{amssymb}
\usepackage{amsthm}
\usepackage{enumerate}
\usepackage{centernot}
\usepackage{fullpage}
\usepackage{mathrsfs}
\usepackage{graphicx}
\usepackage[colorlinks,
linkcolor=blue,
anchorcolor=green,
citecolor=blue
]{hyperref}
\usepackage{color}

\theoremstyle{plain}

\newtheorem{proposition}[]{Proposition}
\newtheorem{lemma}[]{Lemma}
\newtheorem{theorem}[]{Theorem}
\newtheorem{corollary}[]{Corollary}

\newtheorem{remark}[]{Remark}

\newcommand{\E}{\mathbb{E}}
\newcommand{\Z}{\mathbb{Z}}
\newcommand{\R}{\mathbb{R}}
\newcommand{\Cov}{\operatorname{Cov}}
\newcommand{\cone}{\operatorname{cone}}
\newcommand{\1}{\mathbf{1}}
\newcommand{\ind}{\mathrel{\leftrightarrow}}
\newcommand{\notind}{\mathrel{\nleftrightarrow}}

\newcommand{\norminf}[1]{\lVert #1\rVert_\infty}
\title[Connection inequality]{A Connection Inequality for the Random-Cluster Model}

\author{Jianping Jiang}
\address{Yau Mathematical Sciences Center, Tsinghua University, Beijing 100084, China.}
\email{jianpingjiang@tsinghua.edu.cn}

\author{Sike Lang}
\address{Qiuzhen College, Tsinghua University, Beijing 100084, China.}
\email{langsk24@mails.tsinghua.edu.cn}

\author{Yuxuan Sun}
\address{Qiuzhen College, Tsinghua University, Beijing 100084, China.}
\email{yx-sun23@mails.tsinghua.edu.cn}

\begin{document}
\begin{abstract}
	For a finite graph $G=(V,E)$, let $\phi$ be the random-cluster measure with edge parameters $(p_e)_{e\in E}$, cluster weight $q\geq 1$ and free boundary condition. For each $x \in V$ and all nonempty $ A, B \subseteq V$, we prove
	\[\phi(x \leftrightarrow B) \geq \phi(x \leftrightarrow A) \min_{a \in A}\phi(a \leftrightarrow B).\]
	For $q=1$, this is the (slightly generalized) Kozma-Nitzan inequality. Our extension to $q>1$ is based on a recent proof of this inequality by Blanc-Renaudie. As an application, we give a simple proof of sharp upper bounds for one-arm and point-to-plane connection probabilities in the subcritical regime. 
\end{abstract}
\maketitle

\section{The connection inequalities}
Let $G=(V,E)$ be a finite graph. The Fortuin-Kasteleyn (FK) random-cluster measure with edge parameters $p_e \in [0,1]$ for each $e\in E$, cluster weight $q\geq 1$ and free boundary condition is
\[\phi(\omega)=\frac{1}{Z}q^{\kappa(\omega)} \prod_{e \in E} p_e^{\omega_e} (1-p_e)^{1-\omega_e}, \qquad \omega \in \{0,1\}^E,\]
where $\kappa(\omega)$ is the number of clusters in the graph with vertex set $V$ and edge set $\{e \in E: \omega_e=1\}$ and $Z$ is the partition function. An edge $e$ is \textbf{open} if $\omega_e=1$; otherwise it is \textbf{closed}. For $x,y \in V$, write $x \ind y$ for the event that there is an open path in $G$ that connects $x$ and $y$. For any $A,B\subseteq V$, write $A \ind B$ if there exist $x\in A$ and $y \in B$ such that $x \ind y$. $A \notind B$ denotes the complement of $A \ind B$.

We also define
\[
 C_x(\omega):=\{y\in V: y\ind x\text{ in }\omega\},\qquad
 C_A(\omega):=\bigcup_{x\in A}C_x(\omega).
\]

A function $F:2^V\to\R$ is \textbf{increasing} if $F(A)\leq F(B)$ for any $A\subseteq B \subseteq V$. Write $\1_E$ for the indicator of an event $E$. Let $\E$ be the expectation with respect to $\phi$. Our main result is the following cluster function inequality.

\begin{theorem}\label{thm:conditional}
Let $A\subseteq V$ be nonempty and $x\in V$. 
For every increasing function
$F:2^V\to\R$ and every
\[
 a^*\in
 \operatorname*{arg\,min}_{a\in A}\E[F(C_a)],
\]
we have
\begin{equation}\label{eq:conditional-cluster-comparison}
 \E\!\left[
   \bigl(F(C_x)-F(C_{a^*})\bigr)
   \1_{\{x\ind A\}}
 \right]
 \geq0.
\end{equation}
\end{theorem}

If we take $F(C):=\1_{C \cap B \neq \emptyset}$, then Theorem \ref{thm:conditional} and the FKG inequality give
\begin{corollary}\label{cor:connection}
For every nonempty $A\subseteq V$, every $B\subseteq V$, and every
$x\in V$,
\begin{equation}\label{eq:strong}
 \phi(x\ind A,\ x\ind B)
 \geq
 \phi(x\ind A)\min_{a\in A}\phi(a\ind B).
\end{equation}
In particular,
\begin{equation}\label{eq:connection}
 \phi(x\ind B)
 \geq
 \phi(x\ind A)\min_{a\in A}\phi(a\ind B).
\end{equation}
\end{corollary}
The inequality \eqref{eq:connection} with $B$ a singleton and $q=1$ is Conjecture 1 in Kozma and Nitzan~\cite{KN24}, who reduced the proof of $\theta(p_c)=0$ to this inequality. The Kozma-Nitzan inequality was proved first by AI~\cite{Led26,Bou26}. Later, Blanc-Renaudie~\cite{BR26} presented a shorter, more accessible proof. The main goal of this paper is to generalize the proof for $q=1$ in~\cite{BR26} to all $q >1$. We also aim to make the proof easy to understand and we think the current proof is arguably simpler than the previous ones. It is conceivable that Corollary~\ref{cor:connection} could be useful in studying whether there is percolation at criticality. For $d=2$, it was proved that the phase transition is continuous if $1\leq q \leq 4$~\cite{DCST17} and discontinuous if $q>4$ \cite{DCGHMT21}. For $d \geq 3$, it is known that the phase transition is continuous if $q=2$~\cite{ADCS15} and discontinuous if $q$ is large enough~\cite{MMSRS91}.

As one application, we prove the following upper bounds for point-to-plane and one-arm connection probabilities.
\begin{proposition}
\label{prop:point-plane}
Let $\phi_{\Z^d}=\phi^0_{\Z^d,p,q}$ be the homogeneous nearest-neighbor free
random-cluster measure on $\Z^d$, $d\geq1$, with $p\in(0,1)$ and $q\geq1$.
Set
\[
 \tau(y):=\phi_{\Z^d}(0\ind y),
 \qquad
 P_n:=\{y\in\Z^d:y_1=n\},
\]
and
\[
 m:=
 -\lim_{k\to\infty}\frac1k\log\tau(ke_1) \text{ with }e_1:= \text{the first coordinate unit vector}.
\]
Then,
\begin{equation}\label{eq:hyperplane}
 \phi_{\Z^d}(0\ind P_n)\leq e^{-mn}, \qquad \forall n \geq 1.
\end{equation}
Moreover, for
\[
 \Lambda_n=[-n,n]^d\cap\Z^d,\qquad \partial \Lambda_n:=\Lambda_n \setminus \Lambda_{n-1},
\]
\begin{equation}\label{eq:one-arm}
 \phi_{\Z^d}(0\ind\partial\Lambda_n)
 \leq2d\,e^{-mn}, \qquad \forall n \geq 1.
\end{equation}
\end{proposition}
For $q=1$, bounds of this form were a main input in~\cite{DL26} in their study of first-passage percolation; see Propositions 2.3 and 2.4 in~\cite{DL26}. They proved the upper bounds using Ornstein-Zernike behavior for subcritical Bernoulli percolation from \cite{CCC91}. For general $q\geq 1$, the asymptotic behavior of the one-arm probability was established in Corollary 1.1 of \cite{CIV08} using Ornstein-Zernike decay. Proposition~\ref{prop:point-plane} gives the sharp upper bounds (up to a multiplicative constant) by a simple application of the connection inequality.

In the paper, we assume $p_e \in (0,1)$ for each $e \in E$ so that all avoidance events used for conditioning have positive probability. It is clear that Theorem~\ref{thm:conditional} and Corollary~\ref{cor:connection} also hold for all $p_e\in [0,1]$ by continuity (see Remark~\ref{rem:continuity}).

We briefly describe the structure of the paper.
Section~\ref{sec:deletion} is devoted to handling the dependence present when $q>1$; we use exploration coupling to show that certain centered cluster observables remain monotone when additional vertices are deleted. The main proof in Sections~\ref{sec:conditional} and~\ref{sec:comparison} relies on Blanc-Renaudie's proof \cite{BR26}.
In Section~\ref{sec:conditional},  we use the monotonicity obtained in Section~\ref{sec:deletion} to prove that the conditional covariance vectors associated with an ordered family of marked vertices belong to a suitable nonnegative cone.

Choose distinct marked vertices $v_1,\ldots,v_n$. For $1\leq i,j\leq n$, define the \textbf{connection probabilities}
\[
 h_j(i):=\phi(v_i\ind v_j\mid v_j \notind \{v_1,\dots,v_{j-1}\}).
\]
Let $H:=(h_1\ \cdots\ h_n)$ and $M:=H^{-1}$. Since $h_j(i)=0$ for $i<j$ and $h_j(j)=1$, both $H$ and $M$ are unit lower triangular. In Section~\ref{sec:comparison}, we prove
\begin{proposition}\label{prop:comparison-weights}
Let $v_1,\ldots,v_n$ be distinct marked vertices. For  $j<n$,
define
\(c_j:=-M_{nj}\geq 0\). For every increasing
$F:2^V\to\mathbb R$,
\begin{equation}\label{eq:weighted-cluster-comparison}
    \E[F(C_{v_n})]
    \geq
    \sum_{j<n}c_j\,\E[F(C_{v_j})]+r_nF(\emptyset),
\end{equation}
where
\begin{equation}\label{eq:weight-mass}
    r_n:=1-\sum_{j<n}c_j=\phi(v_n\notind \{v_1,\dots,v_{n-1}\}).
\end{equation}
In particular, the weights $c_j$ depend only on the marked vertices
and not on $F$.
\end{proposition}

Theorem~\ref{thm:conditional}, Corollary~\ref{cor:connection} and Proposition~\ref{prop:point-plane} all follow from the inequality \eqref{eq:weighted-cluster-comparison} for some choice of $F$. So we believe this inequality may have further applications.


\section{Deletion monotonicity}\label{sec:deletion}

To prove the deletion monotonicity lemma below, we use three standard properties of the random-cluster model: the domain Markov property, comparison between boundary conditions, and the FKG inequality. Proofs of these properties can be found, for example, in \cite{Gri06} and \cite{DC20}.

\medskip\noindent\textbf{Deleted measures.}
For $K\subseteq V$, let $\phi^K$ be the free FK measure obtained by forcing every edge incident to $K$ to be closed. We write $\E^K$ for expectation under $\phi^K$.

If $K\subseteq L$, then $\phi^L$ is stochastically dominated by $\phi^K$. For $U\subseteq V$, let $\mathcal{F}_U$ be the sigma-algebra generated by an exploration that reveals every edge incident to $C_U$, leaving the edges with both endpoints outside $C_U$ unexamined. By the domain Markov property, for any event $E$
depending only on edges with both endpoints outside $C_U$,
\begin{equation}\label{eq:revealed-clusters}
 \phi(E \mid \mathcal{F}_U)=\phi^{C_U}(E),\qquad
 \phi^K(E\mid \mathcal{F}_U)=\phi^{C_U\cup K}(E),
\end{equation}
almost surely under the respective measures. 

\begin{lemma}[Deletion monotonicity]\label{lem:deletion}
Fix $U\subseteq V$, $b\notin U$, and an increasing $F:2^V\to\mathbb R$. Set
\begin{equation}\label{eq:gZ}
 g(S):=\E^S[F(C_b)],\qquad
 Z:=\1_{\{b\notin C_U\}}\bigl(F(C_b)-g(C_U)\bigr).
\end{equation}
Then
\begin{equation}\label{eq:Z-centered}
 \E[Z\mid \mathcal{F}_U]=0, \qquad \E[Z]=0.
\end{equation}

Moreover, the function
\begin{equation}\label{eq:Psi}
 \Psi(K):=-\E^K[Z]
\end{equation}
is increasing in $K\subseteq V$.
\end{lemma}

\begin{proof}
On $\{b\notin C_U\}$, the cluster $C_b$ lies in the exterior of $C_U$.
Thus \eqref{eq:revealed-clusters} gives
\[
 \1_{\{b\notin C_U\}}\E[F(C_b)\mid \mathcal{F}_U]=\1_{\{b\notin C_U\}}g(C_U),\qquad
 \1_{\{b\notin C_U\}}\E^K[F(C_b)\mid \mathcal{F}_U]=\1_{\{b\notin C_U\}}g(C_U\cup K).
\]
Since $Z=0$ on $\{b\in C_U\}$, the first identity implies
\begin{equation*}
 \E[Z\mid \mathcal{F}_U]
 =\1_{\{b\notin C_U\}}
   \bigl(\E[F(C_b)\mid \mathcal{F}_U]-g(C_U)\bigr)=0,
\end{equation*}
and hence $\E[Z]=0$. The second identity yields
\begin{equation}\label{eq:Psi-identity}
 \Psi(K)=\E^K\!\left[\1_{\{b\notin C_U\}}
 \bigl(g(C_U)-g(C_U\cup K)\bigr)\right]\geq0,
\end{equation}
where the inequality follows because $g$ is decreasing.

Fix $K\subseteq L$. We construct a coupling of three configurations:
$\omega_L\sim\phi^L$, $\omega_K\sim\phi^K$, and an auxiliary $\eta$.
Write $S_L:=C_U(\omega_L)$ and $S_K:=C_U(\omega_K)$.
For the remainder of this proof, $\hat{\E}$ denotes expectation under this coupling. At each step, query the same edge in all active configurations. Each query uses a fresh common uniform variable and each configuration's own opening probability conditioned on exposed edges.

First explore $S_L$ in $\omega_L$, revealing the same edges in $\omega_K$, until all edges incident to $S_L$ are revealed. We then add a new configuration $\eta$ with law $\phi^{S_L}$, and explore $S_K$. Once all edges incident to $S_K$ are revealed, complete the still-unqueried
edges with both endpoints outside $S_K$ in any order by the same rule.

After all edges incident to $S_L$ have been explored, the remaining exterior law of $\omega_L$ is $\phi^{S_L\cup L}$, dominated by the law $\phi^{S_L}$ used for $\eta$.
Comparison between boundary conditions gives
\begin{equation}\label{eq:coupling-order}
 S_L\subseteq S_K,\ 
 \omega_L\leq\omega_K \text{ on all edges},\ 
 \omega_L\leq\eta \text{ on edges with both endpoints in } V \setminus S_L.
\end{equation}


We can now compare $\Psi(L)$ and $\Psi(K)$. Put
\[
 X_L:=F(C_b(\omega_L)),\qquad X_K:=F(C_b(\omega_K)),\qquad
 Y:=F(C_b(\eta)).
\]
The conditional law of $\eta$ gives $\hat{\E}[Y\mid S_L]=g(S_L)$. Since $\omega_L\sim\phi^L$,
the definition of $\Psi$ gives
\begin{align}
 \Psi(L)
 &=\E^L\bigl[\1_{\{b\notin C_U\}}(g(C_U)-F(C_b))\bigr]\nonumber\\
 &=\hat{\E}\bigl[\1_{\{b\notin S_L\}}(g(S_L)-X_L)\bigr]\nonumber\\
 &=\hat{\E}\bigl[\1_{\{b\notin S_L\}}(Y-X_L)\bigr].\label{eq:PsiL}
\end{align}
By \eqref{eq:coupling-order}, $X_K\geq X_L$ everywhere and $Y\geq X_L$
on $\{b\notin S_L\}$. Together with $S_L\subseteq S_K$, this gives the
pointwise comparison
\begin{equation}\label{eq:comparisonXandY}
    \1_{\{b\notin S_L\}}(Y-X_L)
 \geq\1_{\{b\notin S_K\}}(Y-X_K).
\end{equation}

Recall that after all edges incident to $S_K$ have been explored in all three configurations, the remaining exterior laws for $\omega_L$ and $\omega_K$ are $\phi^{S_K \cup L}$ and $\phi^{S_K \cup K}$ respectively; the remaining exterior law of $\eta$ may have boundary wiring, so it dominates $\phi^{S_K}$. Therefore,
\begin{equation}\label{eq:continuation-means}
 \1_{\{b\notin S_K\}}\hat{\E}[X_K\mid S_K]=\1_{\{b\notin S_K\}}g(S_K\cup K),\qquad
 \1_{\{b\notin S_K\}}\hat{\E}[Y\mid S_K]\geq \1_{\{b\notin S_K\}} g(S_K).
\end{equation}
Plugging \eqref{eq:comparisonXandY} into \eqref{eq:PsiL}, and then applying \eqref{eq:continuation-means}, we obtain
\begin{align*}
 \Psi(L)
 &\geq \hat{\E}\bigl[\1_{\{b\notin S_K\}}(Y-X_K)\bigr]\\
 &\geq \hat{\E} \bigl[\1_{\{b\notin S_K\}}
       (g(S_K)-g(S_K\cup K))\bigr]
 =\Psi(K),
\end{align*}
where we have used \eqref{eq:Psi-identity} in the last equality. Hence $\Psi$ is increasing.
\end{proof}

\section{Conditional covariance cone}\label{sec:conditional}
Choose distinct marked vertices $v_1,\ldots,v_n$. For $1\leq j\leq n$, define
\[
 U_j:=\{v_1,\ldots,v_{j-1}\},\qquad
 D_j:=\{v_j\notind U_j\} \text{ with }D_1:=\{0,1\}^E,\qquad r_j:=\phi(D_j)>0.
\]
This definition of $r_n$ agrees with \eqref{eq:weight-mass}. For $1\leq i,j\leq n$, recall that the connection probabilities are defined by $h_j(i):=\phi(v_i\ind v_j\mid D_j)$ and define the
\textbf{conditional covariances}
\begin{equation}\label{eq:kappa}
 \kappa_j^F(i):=\Cov\!\left(F(C_{v_j}),\1_{\{v_i\ind v_j\}}\mid D_j\right).
\end{equation}
We regard $h_j$ and $\kappa_j^F$ as column vectors in $\mathbb R^n$.

For vectors $u_1,\ldots,u_m$, write $\cone\{u_1,\ldots,u_m\}$ for their \textbf{nonnegative cone}, namely the set of all nonnegative linear combinations of these vectors.

\begin{lemma}[Conditional covariance cone]\label{lem:cone}
For every increasing $F:2^V\to\mathbb R$ and every $j$,
\begin{equation}\label{eq:cone}
 \kappa_j^F\in\cone\{h_{j+1},\ldots,h_n\},
\end{equation}
where the cone is $\{0\}$ when $j=n$.
\end{lemma}

\begin{proof}
We use descending induction on $j$, simultaneously for all increasing $F$.

For $j=n$, each marked connection indicator is constant on $D_n$, so $\kappa_n^F=0$.

Fix $j<n$ and assume the statement for larger indices. We use $g,Z,\Psi$ defined in \eqref{eq:gZ} and \eqref{eq:Psi} with $b=v_j$, $U=U_j$ and F. For $j\leq i\leq n$, set
\begin{equation}\label{eq:zi}
 R_i:=\{v_j,\ldots,v_i\},\qquad
 z_i(\ell):=\E\bigl[Z\1_{\{v_\ell\ind R_i\}}\bigr].
\end{equation}

\smallskip\noindent First, we show that $z_j\in\cone\{h_{j+1},\ldots,h_n\}$. \\
We claim
\begin{equation}\label{eq:z-sign}
 z_i(\ell)\geq0\quad\text{for all }i,\ell,\qquad z_n=0.
\end{equation}
Indeed, conditional on $\mathcal{F}_{U_j}$, $Z$ is either zero if $v_j \in C_{U_j}$ or a centered increasing function of the free exterior configuration if $v_j \notin C_{U_j}$. The connection indicator $\1_{\{v_\ell\ind R_i\}}$ is either
determined or increasing in that exterior, so the FKG inequality and Lemma~\ref{lem:deletion} give
\[z_i(\ell)=\E\bigl[\E[Z\1_{\{v_\ell\ind R_i\}} \mid \mathcal{F}_{U_j}]\bigr] \geq \E\bigl[\E[Z \mid \mathcal{F}_{U_j}]\E[\1_{\{v_\ell\ind R_i\}} \mid \mathcal{F}_{U_j}]\bigr]=0.\]
For $i=n$, the indicator is $\mathcal{F}_{U_j}$-measurable
when $\ell<j$ and equals $1$ when $\ell\geq j$; Lemma~\ref{lem:deletion} gives
\begin{equation*}
    z_n(\ell)=\begin{cases}
        \E\bigl[\E[Z\1_{\{v_\ell\ind R_n\}} \mid \mathcal{F}_{U_j}]\bigr]=\E\bigl[\1_{\{v_\ell\ind R_n\}}\E[Z \mid \mathcal{F}_{U_j}]\bigr]=0, & \ell<j,\\
        \E[Z]=0,& \ell \geq j.
    \end{cases}
\end{equation*}
This proves the claim.

For $j\leq i<n$, adding $v_{i+1}$ to $R_i$ increases the indicator by
$\1_{\{v_{i+1}\notind R_i\}}\1_{\{v_\ell\ind v_{i+1}\}}$.
Also $\E[Z]=0$ gives $z_i(i+1)=- \E[Z \1_{v_{i+1} \notind R_i}]$. Hence
\begin{align}
 z_i(\ell)-z_{i+1}(\ell)-z_i(i+1)h_{i+1}(\ell)
 &=-\E\bigl[Z\1_{\{v_{i+1}\notind R_i\}}(\1_{\{v_\ell\ind v_{i+1}\}}-h_{i+1}(\ell))\bigr]\notag\\
 &=-\E\bigl[Z\1_{D_{i+1}}(\1_{\{v_\ell\ind v_{i+1}\}}-h_{i+1}(\ell))\bigr],\label{eq:z-calculation}
\end{align}
where the last equality follows because $\1_{\{v_{i+1}\ind U_j\}} \1_{\{v_{i+1} \notind R_i\}} \1_{\{v_{\ell} \ind v_{i+1}\}}$ is $\mathcal{F}_{U_j}$-measurable, and $Z$ has conditional mean zero by Lemma~\ref{lem:deletion}.

On $D_{i+1}$, the cluster $C_{v_{i+1}}$ avoids $U_j\cup R_i$. Conditional on $C_{v_{i+1}}=K$, the domain Markov property gives a free FK law on the exterior and hence $\E[Z\mid C_{v_{i+1}}=K]=\E^K[Z]=-\Psi(K)$. Thus \eqref{eq:z-calculation} gives
\begin{equation}\label{eq:z-recursion}
 z_i=z_{i+1}+z_i(i+1)h_{i+1}+r_{i+1}\kappa_{i+1}^{\Psi}.
\end{equation}
Since $\Psi$ is increasing by Lemma~\ref{lem:deletion}, the induction hypothesis applies to $\kappa_{i+1}^{\Psi}$. Backward iteration of \eqref{eq:z-recursion} from $z_n=0$ and \eqref{eq:z-sign} yield
\begin{equation}\label{eq:z-cone}
 z_j\in\cone\{h_{j+1},\ldots,h_n\}.
\end{equation}

\smallskip\noindent Next, we recover the covariance $\kappa_j^F$ from $z_j$, thus establishing \eqref{eq:cone}.\\
 For this, define an operator on functions of vertex sets by
\begin{equation}\label{eq:T}
 (TF)(C):=\E^C[g(C_{U_j})]
 =\E^C\!\left[\E^{C_{U_j}}[F(C_{v_j})]\right].
\end{equation}
The map $g$ is decreasing in its deletion set, while enlarging $C$ makes the sampled $C_{U_j}$ stochastically smaller by comparison between boundary conditions. Hence $T$ preserves increasing functions: $TF$ is increasing whenever $F$ is increasing.

We claim
\begin{equation}\label{eq:kappa-decomp}
 \kappa_j^F=\frac{z_j}{r_j}+\kappa_j^{TF}.
\end{equation}
Indeed, on $D_j$ the indicator $\1_{\{v_\ell\ind v_j\}}$ is determined by $C_{v_j}$, and thus
\begin{equation}\label{eq:gtoTF}
    \E[g(C_{U_j}) \1_{\{v_\ell\ind v_j\}} \mid C_{v_j},D_j]=\1_{\{v_\ell\ind v_j\}} (TF)(C_{v_j}).
\end{equation}
Moreover, $\E[F(C_{v_j})-g(C_{U_j})\mid C_{U_j}, D_j]=0$ implies
\begin{equation}\label{eq:F-g}
    \E[F(C_{v_j})-g(C_{U_j})\mid D_j]=0.
\end{equation}
Since $Z=\1_{D_j}(F(C_{v_j})-g(C_{U_j}))$ and $R_j=\{v_j\}$, for each $\ell$ we have
\begin{align*}
 \frac{z_j(\ell)}{r_j}
 &=\E\bigl[(F(C_{v_j})-g(C_{U_j}))\1_{\{v_\ell\ind v_j\}}\mid D_j\bigr]\\
 &=\E\bigl[(F(C_{v_j})-(TF)(C_{v_j}))\1_{\{v_\ell\ind v_j\}}\mid D_j\bigr]\\
 &=\kappa_j^F(\ell)-\kappa_j^{TF}(\ell),
\end{align*}
where we have used \eqref{eq:gtoTF} in the second equality and \eqref{eq:F-g} in the third equality. This proves~\eqref{eq:kappa-decomp} coordinatewise.
Iterating \eqref{eq:kappa-decomp} and applying \eqref{eq:z-cone} to $F,TF,\ldots,T^{m-1}F$ gives
\begin{equation}\label{eq:kappa-telescope}
 \kappa_j^F-\kappa_j^{T^mF}
 =\frac1{r_j}\sum_{t=0}^{m-1}z_j^{T^tF}
 \in\cone\{h_{j+1},\ldots,h_n\}.
\end{equation}
Here $z_j^{T^tF}$ denotes the vector constructed from $T^tF$.

To pass to the limit, note that $T$ may be viewed as a transition kernel: from $C$, sample $K:=C_{U_j}$ under $\phi^C$, then sample a new $v_j$-cluster under $\phi^K$. From every starting set $C$, the kernel outputs $\{v_j\}$ with probability at least
\[
 \varepsilon:=\prod_{e\ni v_j}(1-p_e)>0.
\]
This is because closing all edges incident to $v_j$ suffices, and $\phi(e \text{ is closed} \mid \omega_f, f\neq e) \geq 1-p_e$.
Removing mass $\varepsilon$ at $\{v_j\}$ from each transition law gives, for every $C \subseteq V$,
\[
 (1-\varepsilon)\min_{C'\subseteq V}F(C')
 \leq (TF)(C)-\varepsilon F(\{v_j\})
 \leq (1-\varepsilon)\max_{C'\subseteq V}F(C').
\]
Taking the difference between the maximum and minimum over $C$ and iterating yields
\begin{equation}\label{eq:osc}
 \operatorname{osc}(T^mF)\leq(1-\varepsilon)^m\operatorname{osc}(F),\qquad
 \operatorname{osc}(F):=\max_{C \subseteq V} F(C)-\min_{C \subseteq V} F(C).
\end{equation}
Consequently $\kappa_j^{T^mF}\to0$. The cone generated by finitely many vectors is closed, so \eqref{eq:kappa-telescope} implies \eqref{eq:cone} as $m\to\infty$.
\end{proof}

\section{Comparison weights inequality}\label{sec:comparison}
In this section, we prove Proposition~\ref{prop:comparison-weights}. Recall that $H=(h_1,\ \cdots, h_n), M:=H^{-1}$, and both $H$ and $M$ are unit lower triangular. Writing $e_j$ for the $j$th coordinate
unit vector, we have
\(
    Mh_j=e_j.
\)
Hence Lemma~\ref{lem:cone} gives, coordinatewise,
\begin{equation}\label{eq:M-kappa}
    M\kappa_j^F\geq0
    \qquad
    \text{for every increasing real-valued }F.
\end{equation}

To determine the signs of entries of $M$, take
\[
    F_0(C):=\1_{\{|C|\geq2\}},
    \qquad
    s_j:=\phi(C_{v_j}=\{v_j\}\mid D_j)>0.
\]
A direct calculation gives
\[
    \kappa_j^{F_0}=s_j(h_j-e_j).
\]
Therefore, by \eqref{eq:M-kappa},
\[
    0\leq M\kappa_j^{F_0}
    =s_j(e_j-Me_j),
\]
and hence
\begin{equation}\label{eq:M-sign}
    M_{ii}=1,
    \qquad
    M_{ij}\leq0
    \quad(i\ne j).
\end{equation}

\begin{proof}[Proof of Proposition \ref{prop:comparison-weights}]
We first assume $F(\emptyset) = 0$.
Write
\[
    f_i:=\E[F(C_{v_i})],
    \qquad
    f:=(f_1,\ldots,f_n)^T,
    \qquad
    m_j:=\E[F(C_{v_j})\mid D_j]\geq0.
\]
For each fixed $i$, the events
\(
    D_j\cap\{v_i\ind v_j\},
     1\leq j\leq n,
\)
partition the sample space. On that event,
$C_{v_i}=C_{v_j}$. Therefore, by the definition of conditional covariance, for each $i$,
\[
    f_i= \sum_{j=1}^n r_j\,
       \E\!\left[F(C_{v_j})\mathbf{1}_{\{v_i\ind v_j\}}
       \,\middle|\,D_j\right]
       = \sum_{j=1}^n r_j
       \bigl(m_jh_j(i)+\kappa_j^F(i)\bigr).
\]
In vector form, this gives
\begin{equation}\label{eq:first-marked}
    f=\sum_{j=1}^n r_j\bigl(m_jh_j+\kappa_j^F\bigr).
\end{equation}
Multiplying by $M$ and using \eqref{eq:M-kappa},
\[
    Mf
    =
    \sum_{j=1}^n
    r_j\bigl(m_je_j+M\kappa_j^F\bigr)
    \geq0.
\]
Taking the last coordinate and using $M_{nn}=1$ gives
\[
    f_n+\sum_{j<n}M_{nj}f_j\geq0,
\]
which is exactly \eqref{eq:weighted-cluster-comparison}.

Finally, taking $F\equiv1$ in \eqref{eq:first-marked} (which is valid for any $F$) gives
\[
    \1=Hr,
    \qquad
    \text{ hence }M\1=r,
    \qquad
    r:=(r_1,\ldots,r_n)^T.
\]
The last coordinate gives

\[
    \sum_{j<n}c_j
    =1-r_n=1-\phi(D_n)
    =\phi(v_n\ind U_n).
\]
This completes the proof of the proposition under the assumption $F(\emptyset) = 0$. For general $F$, we may apply \eqref{eq:weighted-cluster-comparison} to $\tilde{F}:=F-F(\emptyset)$ to conclude the proof.
\end{proof}

\section{Proof of the main theorem}
In this section, we prove our main results on connection inequalities.

\begin{proof}[Proof of Theorem~\ref{thm:conditional}]
If $x\in A$, then the theorem holds trivially. So we assume $x \notin A$. Let
\[
    A:=\{v_1,\ldots,v_{n-1}\},
    \qquad
    v_1:=a^*,
    \qquad
    v_n:=x,
\]
and let $c_1,\ldots,c_{n-1}$ be the corresponding comparison weights defined in Proposition~\ref{prop:comparison-weights}.
Define
\[G(K):=\begin{cases}
    \E^K[F(C_{a^*})],& a^* \notin K\\
    F(\emptyset)-F(V)+F(K), & a^* \in K.
\end{cases}
\]
Let
\[\tilde{F}(K):=\left[F(K)-G(K)+F(V)-F(\emptyset)\right] \1_{\{K \cap A \neq \emptyset \}}.\]
Let $D:=F(V)-F(\emptyset)$. On sets avoiding $a^*$, the bracket $F(K)- \E^K[F(C_{a^*})]+D$ is increasing and belongs to $[0,2D]$; on sets containing $a^*$, $\tilde{F}(K)=2D$. So $\tilde{F}$ is nonnegative and increasing. Proposition~\ref{prop:comparison-weights} implies that
\begin{equation}\label{eq:F-gineq}
    \E\bigl[\left(F(C_{v_n})- G(C_{v_n})\right) \1_{\{C_{v_n}\cap A \neq \emptyset\}} \bigr] \geq \sum_{j<n} c_j \E\bigl[\left(F(C_{v_j})- G(C_{v_j})\right) \1_{\{C_{v_j}\cap A \neq \emptyset\}} \bigr],
\end{equation}
where we have used \eqref{eq:weight-mass} to cancel the contribution from $[F(V)-F(\emptyset)]\1_{\{K \cap A \neq \emptyset \}}$. By conditioning on $C_{v_j}$, it is easy to check that
\[\E\bigl[\E^{C_{v_j}}[F(C_{a^*})] \1_{\{C_{v_j}\cap A \neq \emptyset, v_j \notind a^*\}}\bigr]=\E[F(C_{a^*})\1_{\{C_{v_j}\cap A \neq \emptyset, v_j \notind a^*\}}],\qquad 1\leq j \leq n.\]
Therefore,
\[\E[G(C_{v_j})\1_{\{C_{v_j}\cap A \neq \emptyset\}}]=\E[F(C_{a^*})\1_{\{C_{v_j}\cap A \neq \emptyset\}}] + [F(\emptyset)-F(V)]\phi(v_j\ind a^*), \qquad 1 \leq j \leq n.\]
Plugging these into \eqref{eq:F-gineq}, we get
\begin{equation}\label{eq:Fdiffatxanda*}
    \E\bigl[\left(F(C_x)-F(C_{a^*})\right) \1_{\{x\ind A\}}\bigr] \geq \sum_{j<n} c_j \E\bigl[\left(F(C_{v_j})-F(C_{a^*})\right) \bigr],
\end{equation}
where we have used $h_1(n)-\sum_{j<n}c_jh_1(j)=0$ to cancel the contribution involving $F(\emptyset)-F(V)$; the latter equality is  $Mh_1=e_1$. This completes the proof of the theorem since for each $j<n$, $c_j \geq 0$ and $\E[F(C_{v_j})]\geq \E[F(C_{a^*})]$ by the definition of $a^*$.
\end{proof}
\begin{remark}\label{rem:continuity}
    Note that \eqref{eq:Fdiffatxanda*} holds with $a^*$ replaced by any fixed $a\in A$. Therefore, 
    \begin{equation}\label{eq:Fdiffatxanda}
        \E\bigl[\left(F(C_x)-F(C_{a})\right) \1_{\{x\ind A\}}\bigr] 
        \geq \min_{b \in A}\E[F(C_{b})]-\E[F(C_{a})],
    \end{equation}
    where the last inequality follows from $\sum_{j<n} c_j \in [0,1]$. By continuity, choosing $a$ to be any minimizer at the limiting parameter vector extends Theorem~\ref{thm:conditional} to all edge parameters in $[0,1]$.
    
    
\end{remark}

Although it is clear that Corollary~\ref{cor:connection} follows directly from Theorem~\ref{thm:conditional}, we give a separate proof using Proposition~\ref{prop:comparison-weights} here. 
\begin{proof}[Proof of Corollary~\ref{cor:connection}]
The case $x\in A$ is immediate. Assume $x\notin A$, and write
\[
    A:=\{v_1,\ldots,v_{n-1}\},
    \qquad
    v_n:=x.
\]
Define
\[
    F(C)
    :=
    \1_{\{C\cap A\ne\varnothing\}}
    \1_{\{C\cap B\ne\varnothing\}}.
\]
Then $F$ is increasing. Proposition~\ref{prop:comparison-weights} implies
\begin{align*}
      \phi(x\ind A,\ x\ind B)
    &\geq
    \sum_{j<n}c_j\phi(v_j\ind B) \notag\\
    &\geq
    \phi(x\ind A)
    \min_{a\in A}\phi(a\ind B). \qedhere
\end{align*}
\end{proof}

\section{Application of the inequality}

We conclude with a simple application to the point-to-plane connection
probability. 
We first record the following weighted consequence.

\begin{lemma}\label{lem:weighted}
Recall that $\phi_{\Z^d}$ is the free infinite-volume random-cluster measure on $\Z^d$ with $p\in (0,1)$ and $q\geq1$.
For $w_y\geq0$, set
\[
 S_w(z):=\sum_{y\in \Z^d}w_y\phi_{\Z^d}(z\ind y).
\]
Then, for every $x\in \Z^d$ and every finite nonempty $A\subset \Z^d$,
\begin{equation}\label{eq:weighted}
 S_w(x)
 \geq
 \phi_{\Z^d}(x\ind A)\min_{a\in A}S_w(a).
\end{equation}
\end{lemma}

\begin{proof}
Passing Proposition~\ref{prop:comparison-weights} through a finite free
exhaustion gives $c_a\geq0$ such that
\[
 \sum_{a\in A}c_a=\phi_{\Z^d}(x\ind A),
 \qquad
 \E_{\Z^d}[F(C_x)]
 \geq
 \sum_{a\in A}c_a\E_{\Z^d}[F(C_a)],
\]
whenever $F$ is nonnegative and increasing. Taking
\[
 F(C)=\sum_{y\in C}w_y,
\]
we obtain
\[
 S_w(x)
 \geq
 \sum_{a\in A}c_aS_w(a)
 \geq
 \phi_{\Z^d}(x\ind A)\min_{a\in A}S_w(a).\qedhere
\]
\end{proof}

We are ready to prove Proposition~\ref{prop:point-plane}.
\begin{proof}[Proof of Proposition~\ref{prop:point-plane}]
By FKG and translation invariance,
\[
 \tau(u+v)\geq\tau(u)\tau(v),
\]
so Fekete's lemma gives
\[
 m=\inf_{k\geq1}
 \frac{-\log\tau(ke_1)}{k},
 \qquad
 \tau(ke_1)\leq e^{-mk}.
\]
By reflection symmetry and FKG,
\[
 \tau(y)^2
 \leq
 \tau(2\norminf{y}e_1)
 \leq
 e^{-2m\norminf{y}},
\]
hence
\[
 \tau(y)\leq e^{-m\norminf{y}}.
\]

Proposition~\ref{prop:point-plane} is trivial if $m=0$. So we may assume $m>0$. For $0<s<m$, set
\[
 M_1(s):=\sum_{y\in\Z^d}e^{sy_1}\tau(y)<\infty.
\]
Taking $w_y=e^{sy_1}$ and $x=0$ in Lemma~\ref{lem:weighted} gives, for every
finite nonempty $A\subset P_n$,
\[
 M_1(s)
 \geq
 \phi_{\Z^d}(0\ind A)e^{sn}M_1(s),
\]
and therefore
\[
 \phi_{\Z^d}(0\ind A)\leq e^{-sn}.
\]
Letting $A\uparrow P_n$ and $s\uparrow m$ proves
\eqref{eq:hyperplane}.

Finally,
\[
 \{0\ind\partial\Lambda_n\}
 \subseteq
 \bigcup_{i=1}^d
 \bigl(
   \{0\ind P_n^{(i)}\}
   \cup
   \{0\ind P_{-n}^{(i)}\}
 \bigr),
\]
where $P_{\pm n}^{(i)}=\{y:y_i=\pm n\}$.
Thus lattice symmetry and the union bound give
\[
 \phi_{\Z^d}(0\ind\partial\Lambda_n)
 \leq
 2d\,e^{-mn}. \qedhere
\]
\end{proof}

\section*{Acknowledgments}
This research was partially supported by the National Natural Science Foundation of China (No. 12271284 and No. 12226001). The authors acknowledge the use of ChatGPT as an auxiliary tool in proof development. The authors take full responsibility for the content.

\bibliographystyle{abbrv}
\bibliography{reference}

\end{document}